\documentclass{amsart}
\usepackage{bm,color}

\title
[Infinitesimal generators ...]
{
Infinitesimal generators \\ of invertible evolution families
}

\author{Yoritaka Iwata}

\address[Yoritaka Iwata]{ 
Institute of Innovative Research, Tokyo Institute of Technology;  \linebreak
Department of Mathematics, Shibaura Institute of Technology.}
\email{iwata\_phys@08.alumni.u-tokyo.ac.jp}

\thanks{The author is grateful to Prof. Emeritus Hiroki Tanabe for valuable comments.}

\theoremstyle{plain}
\newtheorem{theorem}{Theorem}[section]
\newtheorem{corollary}[theorem]{Corollary}
\newtheorem{lemma}[theorem]{Lemma}

\begin{document}

\begin{abstract}
A logarithm representation of operators is introduced as well as a concept of pre-infinitesimal generator.
Generators of invertible evolution families are represented by the logarithm representation, and a set of operators represented by the logarithm is shown to be associated with analytic semigroups.
Consequently generally-unbounded infinitesimal generators of invertible evolution families are characterized by a convergent power series representation.
\end{abstract}

\maketitle
\section{Introduction}  \label{sec1}
The logarithm of an injective sectorial operator was introduced by Nollau~\cite{69nollau} in 1969.
After a long time, the logarithm of sectorial operators were studied again from 1990's \cite{boyadzhiev,00okazawa-01, 00okazawa-02}, and its utility was established with respect to the definition of the logarithms of operators~\cite{03hasse,martinez} (for a review of sectorial operators, see Hasse \cite{06hasse}).
While the sectorial operator has been a generic framework to define the logarithm of operators, the sectorial property is not generally satisfied by the evolution family of operators, where the evolution family correspond to the exponentials of operators in abstract Banach space framework (for the definition of evolution family in this paper, see Sec.~\ref{secevolv}).

In this paper we characterize infinitesimal generators of invertible evolution families that has not been settled so far.
First of all, by introducing a kind of similarity transform, a logarithm representation is obtained for such generators.
The logarithm representation is utilized to show a convergent power series representation of invertible evolution families generated by certain unbounded operators, although the validity of such a representation is not established for any evolution families generated by unbounded operators. 
In this context, the concept of pre-infinitesimal generator is introduced.

\section{Mathematical settings} \label{tp-group}
\subsection{Evolution family on Banach spaces} \label{secevolv}
Let $X$ and $B(X)$ be a Banach space with a norm $\| \cdot \|$ and a space of bounded linear operators on $X$.
The same notation is used for the norm equipped with $B(X)$, if there is no ambiguity.

 For a positive and finite $T$, let elements of evolution family $\{U(t,s) \}_{-T \le t, s  \le T}$ be mappings: $(t,s) \to U(t,s)$ satisfying the strong continuity for $-T \le t, s  \le T$ (for reviews or textbooks, see \cite{02arendt,66Kato,72krein,83pazy,79tanabe}). 
The semigroup properties:
\begin{equation} \label{sg1} 
U(t,r)~ U(r,s) = U(t,s), 
\end{equation}
and
\begin{equation}  \label{sg2} 
U(s,s) = I, 
\end{equation}
are assumed to be satisfied, where $I$ denotes the identity operator of $X$.
Both $U(t,s)$ and $U(s,t)$ are assumed to be well-defined to satisfy 
\begin{equation} \label{sg3}
U(s,t) ~ U(t,s) = U(s,s) = I, 
\end{equation}
where $U(s,t)$ corresponds to the inverse operator of $U(t,s)$. 
Since $U(t,s) ~ U(s,t) = U(t,t) = I$ is also true, the commutation between $U(t,s)$ and $U(s,t)$ follows. 
Operator $U(t,s)$ is a generalization of exponential function; indeed the properties shown in Eqs.~\eqref{sg1}-\eqref{sg3} are satisfied by taking $U(t,s)$ as $e^{t-s}$.
Evolution family is an abstract concept of exponential function valid for both finite and infinite dimensional Banach spaces.

Let $Y$ be a dense Banach subspace of $X$, and the topology of $Y$ be stronger than that of $X$.
The space $Y$ is assumed to be $U(t,s)$-invariant.
Following the definition of $C_0$-(semi)group (cf. the assumption $H_2$ in Sec.~5.3 of Pazy~\cite{83pazy} or corresponding discussion in Kato~\cite{70kato,73kato}), $U(t,s)$ is assumed to satisfy the boundedness \cite{66Kato,83pazy}; there exist real numbers $M$ and $\beta$ such that
\begin{equation} \label{qb} 
\| U(t,s)  \|_{B(X)} \le M e^{\beta t},
\quad
\| U(t,s)  \|_{B(Y)} \le M e^{\beta t}.
\end{equation}
Inequalities \eqref{qb} are practically reduced to
\[ \| U(t,s)  \|_{B(X)} \le M e^{\beta T}, \quad \| U(t,s)  \|_{B(Y)} \le M e^{\beta T},  \]
when $t$-interval is restricted to be finite $[-T,T]$.

\subsection{Pre-infinitesimal generator}
The counterpart of the logarithm in the abstract framework is introduced.
There are two concepts associated with the logarithm of operators; one is the pre-infinitesimal generator and the other is $t$-differential of $U(t,s)$.
For $-T \le t, s  \le T$, the weak limit
\begin{equation}  \label{weakdef} \begin{array}{ll}
 \mathop{\rm wlim}\limits_{h \to 0} h^{-1} (U(t+h,s) - U(t,s)) ~u 
= \mathop{\rm wlim}\limits_{h \to 0}   h^{-1}(U(t+h,t) - I) ~ U(t,s) ~u,  
\end{array} \end{equation}
is assumed to exist for $u$, which is an element of a dense subspace $Y$ of $X$.
A linear operator $A(t): Y  ~\to~  X$ is defined by
\begin{equation} \label{pe-group}
A(t) u := \mathop{\rm wlim}\limits_{h \to 0}  h^{-1} (U(t+h,t) - I) u
\end{equation}
for $u \in Y$ and $-T \le t, s  \le T$, and then let $t$-differential of $U(t,s)$ in a weak sense be
\begin{equation} \label{de-group} \begin{array}{ll} 
\partial_t U(t,s)~u = A(t) U(t,s) ~u.
 \end{array}  \end{equation}
Equation~\eqref{de-group} is regarded as a differential equation
satisfied by $U(t,s) u$ that implies a relation between $A(t)$ and the logarithm:
\[ \begin{array}{ll}
A(t)  = \partial_t U(t,s) ~ U(s,t).
\end{array} \]
Let us call $A(t)$ defined by Eq.~\eqref{pe-group} for a whole family $\{U(t,s)\}_{-T \le t,s \le T}$ the pre-infinitesimal generator. 
Note that pre-infinitesimal generators are not necessarily infinitesimal generators; e.g., in $t$-independent cases, $A$ defined by  Eq.~\eqref{pe-group} is not necessarily a densely-defined and closed linear operator, while $A$ must be a densely-defined and closed linear operator with its
resolvent set included in $\{\lambda \in {\mathbb C}: {\rm Re} \lambda > \beta
\}$.
On the other hand, infinitesimal generators are necessarily pre-infinitesimal generators.
In the following a set of pre-infinitesimal generators is denoted by $G(X)$.

\subsection{A principal branch of logarithm}
The logarithm is defined by the Dunford integral in this paper.
Two difficulties of dealing with logarithm are its singularity at the origin and its multi-valued property.
By introducing a constant $\kappa \in {\mathbb C}$, the singularity can be handled.
Let $\arg$ be a function of complex number, which gives the angle between the positive real axis to the line including the point to the origin.
For the multi-valued property, a principle branch  (denoted by ``Log") of the logarithm (denoted by ``$ \log$") is chosen for any complex number $z \in C$, a  branch of logarithm is defined by
\[ \begin{array}{ll}
 {\rm Log} z  = \log |z| + i \arg Z,  
\end{array} \]
where $Z$ is a complex number chosen to satisfy $|Z| = |z|$, $-\pi < \arg Z \le \pi$, and $\arg Z = \arg z + 2 n \pi$ for a certain integer $n$. 

\begin{lemma}
  \label{lem3}
Let $t$ and $s$ satisfy $0 \le t,s \le T$.
For a given $U(t,s)$ defined in Sec.~\ref{tp-group}, its logarithm is well defined; there exists a certain complex number $\kappa$ satisfying
\begin{equation}
\label{logex3} \begin{array}{ll}
{\rm Log} (U(t,s)+\kappa I) = \frac{1}{2 \pi i} \int_{\Gamma} {\rm Log} \lambda  
 ~ ( \lambda - U(t,s) - \kappa )^{-1}  d \lambda,
\end{array} \end{equation}
where an integral path $\Gamma$, which excludes the origin, is a circle
in the resolvent set of $U(t,s) +\kappa I$.
Here $\Gamma$ is independent of $t$ and $s$. 
${\rm Log} (U(t,s)+ \kappa I)$ is bounded on $X$.
\end{lemma}

\begin{proof}
The logarithm ${\rm Log}$ holds the singularity at the origin, so that it is necessary to show a possibility of taking a simple closed curve (integral path) excluding the origin in order to define the logarithm by means of the Dunford-Riesz integral. 
It is not generally possible to take such a path in case of $\kappa =0$.

First, $U(t,s)$ is assumed to be bounded for $0 \le t,s \le T$ (Eq.~\eqref{qb}), and the spectral set of $U(t,s)$ is a bounded set in ${\mathbb C}$.
Second, for $\kappa$ satisfying
\begin{equation} \label{cr-cond} |\kappa| > M e^{\beta T}, \end{equation}
the spectral set of $U(t,s) + \kappa I$ is separated with the origin.
Consequently it is possible to take an integral path $\Gamma$ including the spectral set of $U(t,s)+\kappa I$ and excluding the origin.
Equation~\eqref{logex3} follows from the Dunford-Riesz integral~\cite{43dunford}.
Furthermore, by adjusting the amplitude of $\kappa$, an appropriate integral path always exists independent of $t$ and $s$.

${\rm Log} (U(t,s)+\kappa I)$ is bounded on $X$, since $\Gamma$ is included in the resolvent set of $(U(t,s)+\kappa I)$.
\end{proof}

According to this lemma, by introducing nonzero $\kappa$, the logarithm
of $U(t,s)+\kappa I$ is well-defined without assuming the sectorial property to $U(t,s)$.
On the other hand Eq.~\eqref{logex3} is valid with $\kappa=0$ only for limited cases.

\section{Main results}
\subsection{Logarithm representation of pre-infinitesimal generator}
\begin{theorem}
\label{thm1}
Let $t$ and $s$ satisfy $-T \le t,s \le T$, and $Y$ be a dense subspace of $X$.
For $U(t,s)$ defined in Sec.~\ref{tp-group}, let $A(t) \in G(X)$ and $\partial_t U(t,s)$ be determined by Eqs.~\eqref{pe-group} and \eqref{de-group} respectively.  
If $A(t)$ and $U(t,s)$ commute, an evolution family $\{ A(t) \}_{-T \le t \le T}$ is represented by means of the logarithm function; there exists a certain complex number $\kappa \ne 0$ such that
\begin{equation} \label{logex} \begin{array}{ll}
 A(t) ~ u =  (I+ \kappa U(s,t))~ \partial_{t} {\rm Log} ~ (U(t,s) + \kappa I) ~ u, 
\end{array} \end{equation}
where $u$ is an element in $Y$.
Note that $U(t,s)$ defined in Sec.~\ref{tp-group} is assumed to be invertible.
\end{theorem}

\begin{proof}
For $U(t,s)$ defined in Sec.~\ref{tp-group}, operators $ {\rm Log} ~ (U(t,s) + \kappa I)$ and $ {\rm Log} ~ (U(t+h,s) + \kappa I)$ are well defined for a certain $\kappa$ (Lemma~\ref{lem3}).
The $t$-differential in a weak sense is formally written by
\begin{equation} \label{difference0} \begin{array} {ll} 
\mathop{\rm wlim}\limits_{h \to 0}  \frac{1}{h} \{ {\rm Log} ~(U(t+h,s)+\kappa I) - {\rm Log} ~(U(t,s)+ \kappa I) \}   \vspace{1.5mm} \\
  =\mathop{\rm wlim}\limits_{h \to 0}   \frac{1}{h} \frac{1}{2 \pi i}
 \int_{\Gamma} {\rm Log} \lambda   \vspace{1.5mm} \\ 
 ~ \{ ( \lambda - U(t+h,s) - \kappa )^{-1}   
  -  ( \lambda - U(t,s) - \kappa )^{-1}  \} d \lambda  \vspace{1.5mm} \\  
 = \mathop{\rm wlim}\limits_{h \to 0}  \frac{1}{2 \pi i}
 \int_{\Gamma} {\rm Log} \lambda   \vspace{1.5mm} \\ 
 ~ \{ (\lambda - U(t+h,s)-\kappa )^{-1} \frac{U(t+h,s)-U(t,s)}{h} (\lambda - U(t,s)- \kappa )^{-1} \} d \lambda   
\end{array} \end{equation} 
where $\Gamma$, which is possible to be taken independent of $t$, $s$ and $h$ for a sufficiently large certain $\kappa$, denotes a circle in the resolvent set of both $U(t,s)+ \kappa I$ and $U(t+h,s)+\kappa I$.
A part of the integrand of Eq.~\eqref{difference0} is estimated as
\begin{equation} \label{intee} \begin{array} {ll}
\quad \| \{ (\lambda - U(t+h,s)-\kappa )^{-1} \frac{U(t+h,s)-U(t,s)}{h} (\lambda - U(t,s)- \kappa )^{-1} \} v \|_X \vspace{1.5mm} \\
  \le  \| \{ (\lambda - U(t+h,s)- \kappa )^{-1} \|_{B(X)}  \vspace{1.5mm} \\
 \|  \frac{U(t+h,s)-U(t,s)}{h} (\lambda - U(t,s)- \kappa )^{-1} \} v \|_X, 
\end{array} \end{equation}
for $v \in X$. 
There are two steps to prove the validity of Eq.~\eqref{difference0}.
In the first step, the former part of the right hand side of Eq.~\eqref{intee} satisfies
\[ \begin{array} {ll}
 \| \{ (\lambda - U(t+h,s)- \kappa )^{-1} \|_{B(X)}  < \infty,
\end{array} \]
since $\lambda$ is taken from the resolvent set of $U(t+h,s)- \kappa I$.
In the same way the operator $(\lambda - U(t,s) - \kappa )^{-1}$ is bounded on $X$ and $Y$.
Then the continuity of the mapping $t \to (\lambda - U(t,s)- \kappa )^{-1}$ as for the strong topology follows:
\[ \begin{array} {ll}
 \|  (\lambda - U(t+h,s)- \kappa )^{-1} - (\lambda - U(t,s)- \kappa I)^{-1}  \|_{B(X)}   \vspace{1.5mm} \\
 \le  \| (\lambda - U(t+h,s)- \kappa )^{-1} \|_{B(X)}   
 \|( U(t+h,s)-U(t,s)) (\lambda - U(t,s)-\kappa )^{-1}  \|_{B(X)}.
\end{array} \]
In the second step, the latter part of the right hand side of Eq.~\eqref{intee} is estimated as
\begin{equation} \label{unibound}  \begin{array}{ll}
 \left\| \frac{U(t+h,s)-U(t,s)}{h} (\lambda  - U(t,s)-\kappa )^{-1}  u  \right\|_X \vspace{1.5mm}\\
 = \left\| \frac{1}{h} \int_t^{t+h} A(\tau) U(\tau,s) (\lambda  - U(t,s)- \kappa )^{-1}  u ~ d\tau  \right\|_X \vspace{1.5mm}\\
 \le  \frac{1}{|h|}  \int_t^{t+h} \| A(\tau)  U(\tau,s)\|_{B(Y,X)}
   \| (\lambda - U(t,s)- \kappa )^{-1}\|_{B(Y)} \|u \|_{Y} ~ d\tau 
\end{array} \end{equation}
for $u \in Y$.
Because $\| A(\tau) U(\tau,s) \|_{B(Y,X)}  < \infty $ is true by assumption, the right hand side of Eq.~\eqref{unibound} is finite.  
Equation~\eqref{unibound} shows the uniform boundedness with respect to $h$, then the uniform convergence ($h \to 0$) of Eq.~\eqref{difference0} follows.
Consequently the weak limit process $h \to 0$ for the integrand of Eq. (\ref{difference0}) is justified, as well as the commutation between the limit and the integral.

According to Eq.~\eqref{difference0}, interchange of the limit with the integral leads to
\[ \begin{array} {ll} 
\partial_t {\rm Log} (U(t,s) + \kappa I) ~ u 
=  \frac{1}{2 \pi i} \int_{\Gamma} d\lambda  \vspace{1.5mm} \\    
 ~ \left[ ( {\rm Log} \lambda )  (\lambda - U(t,s)- \kappa )^{-1} 
 ~\mathop{\rm wlim}\limits_{h \to 0} \left ( \frac{U(t+h,s)-U(t,s)}{h} \right)
~ (\lambda  - U(t,s)- \kappa )^{-1}    \right] ~u \\
\end{array} \] 
for $u \in Y$.
Because we are also allowed to interchange $A(t)$ with $U(t,s)$,
\begin{equation} \label{intermed}  \begin{array}{ll}
\partial_t {\rm Log} (U(t,s) + \kappa I) ~ u \vspace{1.5mm}\\
 = \frac{1}{2 \pi i} \int_{\Gamma} ({\rm Log} \lambda)
 (\lambda-U(t,s)-\kappa )^{-1} 
  A(t) ~ U(t,s) ~ (\lambda -U(t,s)- \kappa )^{-1}  d \lambda ~ u \vspace{1.5mm}\\
 = \frac{1}{2 \pi i} \int_{\Gamma} ({\rm Log} \lambda) ~(\lambda-U(t,s) - \kappa )^{-2} ~ U(t,s) ~ d \lambda ~A(t) ~ u
\end{array} \end{equation}
for $u \in Y$.
A part of the right hand side is calculated as
\[ \begin{array}{ll}
 \quad  \frac{1}{2 \pi i} \int_{\Gamma}~ ({\rm Log} \lambda) ~(\lambda-U(t,s) - \kappa )^{-2} U(t,s) ~ d \lambda  \vspace{1.5mm} \\
 = \frac{1}{2 \pi i} \int_{\Gamma} \frac{1}{\lambda}  ~(\lambda-U(t,s)
 - \kappa )^{-1} ~ U(t,s) ~ d \lambda \vspace{1.5mm} \\
 = \frac{1}{2 \pi i} \int_{\Gamma} \frac{1}{\lambda}  ~(\lambda-U(t,s)- \kappa )^{-1}  ~ \{ \lambda- \kappa  -(\lambda - U(t,s) - \kappa ) \} ~ d \lambda \vspace{1.5mm} \\
 = \frac{1}{2 \pi i} \int_{\Gamma} (\lambda-U(t,s)-\kappa )^{-1} ~ d \lambda
 - \frac{1}{2 \pi i} \int_{\Gamma} \frac{\kappa}{\lambda} (\lambda-U(t,s)-\kappa )^{-1} ~
 d \lambda 
 - \frac{1}{2 \pi i} \int_{\Gamma} \frac{1}{\lambda} ~ d \lambda \vspace{1.5mm} \\
 = \frac{1}{2 \pi i} \int_{\Gamma} (\lambda-U(t,s)-\kappa )^{-1} ~ d \lambda
 - \frac{1}{2 \pi i} \int_{\Gamma} \frac{\kappa}{\lambda} (\lambda-U(t,s)-\kappa )^{-1} ~
 d \lambda \vspace{1.5mm} \\
 =  \frac{1}{2 \pi i} \int_{\Gamma} (\lambda-U(t,s)- \kappa )^{-1} ~ d \lambda   \vspace{1.5mm} \\
~ - \kappa (U(t,s)+ \kappa I )^{-1} \left\{  \frac{1}{2 \pi i} \int_{\Gamma}  \frac{1}{\lambda}  (U(t,s)+ \kappa I) (\lambda - U(t,s)-\kappa)^{-1} d \lambda  \right\}  \vspace{1.5mm} \\
 =    \frac{1}{2 \pi i} \int_{\Gamma} (\lambda-U(t,s)-\kappa )^{-1} ~ d \lambda
  \vspace{1.5mm} \\
~ - \kappa (U(t,s)+\kappa I)^{-1} \left\{  \frac{1}{2 \pi i} \int_{\Gamma} (\lambda - U(t,s)-\kappa )^{-1} d \lambda -\frac{1}{2 \pi i}  \int_{\Gamma} \frac{1}{\lambda} d \lambda  \right\}  \vspace{1.5mm} \\
 =   (I - \kappa (U(t,s)+\kappa I)^{-1} ) ~  \frac{1}{2 \pi i} \int_{\Gamma} (\lambda - U(t,s)-\kappa)^{-1} d \lambda   \vspace{1.5mm} \\
 = (I - \kappa (U(t,s)+ \kappa I)^{-1}) ~ \frac{1}{2 \pi i} \int_{|\nu| = r} \sum_{n=1}^{\infty} \frac{U(t,s)^{n}}{\nu^{n+1}} ~ d \nu     \vspace{1.5mm} \\
  = I- \kappa (U(t,s)+ \kappa I)^{-1},  
\end{array} \] 
due to the integration by parts, where $|\lambda - \kappa| = |\nu| =r$ is a properly chosen circle large enough to include $\Gamma$.
$ (2 \pi i)^{-1}  \int_{\Gamma} \lambda^{-1} d \lambda = 0$ is seen by applying $d {\rm Log} \lambda/d \lambda = 1/\lambda$.  
$(2 \pi i)^{-1} \int_{|\nu| = r} \sum_{n=1}^{\infty} U(t,s)^{n} \nu^{-n-1} ~ d \nu  = I$ follows from the singularity of $\nu^{-n-1}$. 

Consequently we have
\[ \begin{array}{ll}
A(t) ~ u = \{I- \kappa (U(t,s)+\kappa I)^{-1}\}^{-1} ~ \partial_{t} {\rm Log} ~ (U(t,s) + \kappa I) ~ u  \vspace{1.5mm} \\
\quad = (U(t,s)+\kappa I) U(t,s)^{-1} ~ \partial_{t} {\rm Log} ~ (U(t,s) + \kappa I) ~ u  \vspace{1.5mm} \\
\quad =  (I+\kappa U(s,t))~ \partial_{t} {\rm Log} ~ (U(t,s) + \kappa I) ~ u 
\end{array} \]
for $u \in Y$.
\end{proof}

What is introduced by Eq.~\eqref{logex} is a kind of resolvent approximation of $A(t)$
\[ \begin{array}{ll}
 \partial_{t} {\rm Log} ~ (U(t,s) + \kappa I) =   (I+\kappa U(s,t))^{-1} A(t), 
\end{array} \]
in which $A(t)$ is approximated by the resolvent of $U(s,t)$.
As seen in the following it is notable that there is no need to take $\kappa \to 0$.
This point is different from the usual treatment of resolvent approximations.
On the other hand, it is also seen by Eq.~\eqref{logex} that
\[
 \partial_{t} {\rm Log} ~ (U(t,s) + \kappa I) =  (U(t,s)+\kappa I )|_{\kappa = 0} A(t)  (U(t,s)+\kappa I)^{-1}     
\]
shows a structure of similarity transform, where $(U(t,s)+\kappa)|_{\kappa=0}$ means $U(t,s)+\kappa$ satisfying a condition $\kappa=0$.

Under the validity of Theorem~\ref{thm1}, for $-T \le t,s \le T$, let $a(t,s)$ be defined by 
\[  \begin{array}{ll}
a(t,s) := {\rm Log} (U(t,s)+\kappa I),
\end{array} \]
then Eq.~\eqref{logex} is written as $A(t)  =  (I+\kappa U(s,t))~ \partial_{t} a(t,s)$.
Since $\kappa$ is chosen to separate the spectral set of $U(t,s)+\kappa I$ from the origin, the inverse operator of
\[ \begin{array}{ll}
(I+ \kappa U(s,t)) = U(s,t) (U(t,s)+\kappa I) 
\end{array} \]
is well defined as $(I+\kappa U(s,t))^{-1} = (U(t,s)+\kappa I)^{-1} U(t,s)$.
It also ensures that $\partial_{t} a(t,s)$ is well defined.

\begin{corollary} \label{transform}
Let $t$ and $s$ satisfy $0 \le t,s \le T$.
For $U(t,s)$ and $A(t)$ satisfying the assumption of Theorem~\ref{thm1}, the exponential of $a(t,s)$ is represented by a convergent power series:
\begin{equation} \label{convp} \begin{array}{ll} 
e^{a(t,s)} = \sum_{n=0}^{\infty} \frac{ a(t,s)^n}{n!}, \end{array} \end{equation}
with a relation $e^{a(t,s)} = \exp({\rm Log} (U(t,s)+\kappa I)) = U(t,s)+\kappa I$. 
If $a(t,s)$ with different $t$ and $s$ are further assumed to commute,
\begin{equation} \label{replce} \begin{array}{ll}
\partial_t  e^{a(t,s)} u_s =  \partial_t a(t,s)  ~   e^{a(t,s)} u_s  \\
\end{array} \end{equation}
is satisfied for $u_s \in Y$, where $\partial_t$ denotes a $t-$differential in a weak sense.
\end{corollary}

\begin{proof}
Since $a(t,s)$ is a bounded operator on $X$ (Lemma~\ref{lem3}), the exponential of $a(t,s)$ is represented by a convergent power series \cite{82kato}:
\begin{equation} \label{convp} \begin{array}{ll} 
e^{a(t,s)} = \sum_{n=0}^{\infty} \frac{ a(t,s)^n}{n!}, \end{array} \end{equation}
where 
\[ \begin{array}{ll}
 e^{a(t,s)} = \exp({\rm Log} (U(t,s)+\kappa I)) \vspace{1.5mm} \\
 = \frac{1}{2 \pi i} \int_{\Gamma} e^{{\rm Log} \lambda}   
 ~ ( \lambda - U(t,s) - \kappa)^{-1}  d \lambda \vspace{1.5mm} \\
 =  U(t,s)+\kappa I
\end{array} \]  
is satisfied. 
Since $a(t,s)$ with different $t$ and $s$ commute, 
\[ \begin{array}{ll}
\partial_t \{ a(t,s) \}^n 
 = n \{ a(t,s) \}^{n-1} ~ \partial_t a(t,s)  
\end{array} \] 
leads to
\begin{equation} \label{traeq} \begin{array}{ll}
\partial_t e^{a(t,s)} u_s
 = e^{a(t,s)} ~ (\partial_t a(t,s))  u_s 
\end{array} \end{equation}
for $u_s \in Y$.
This is a linear evolution equation satisfied by $e^{a(t,s)}  u_s$.
\end{proof}

Further calculi on Eq.~\eqref{traeq} lead to
\[ \begin{array}{ll}
\partial_t e^{a(t,s)} u_s
 =   \partial_t (U(t,s)+\kappa I) u_s
 =   \partial_t U(t,s) u_s,
\end{array} \]
and 
\[ \begin{array}{ll}
  e^{a(t,s)} ~ (\partial_t a(t,s) ) u_s 
 = (U(t,s)+\kappa I) \partial_t( {\rm Log} (U(t,s)+\kappa I)) u_s \\ 
  = (U(t,s)+\kappa I)  (U(t,s)+\kappa I)^{-1} U(t,s) A(t) u_s \\
 =  U(t,s) A(t) u_s  =  A(t) U(t,s) u_s,
\end{array} \]
where Theorem~\ref{thm1} is applied.
As a result 
\[  \partial_t U(t,s) u_s =   A(t) U(t,s) u_s \]
is obtained.
Note that $e^{a(t,s)}$ does not satisfy the semigroup property, while $U(t,s)$ satisfies it.
 
\section{Abstract Cauchy problem} 
\subsection{Autonomous case} \label{homosection}
Logarithmic representation is utilized to solve autonomous Cauchy problem 
\begin{equation} \label{homoporo} \left\{  \begin{array}{ll}
\partial_t u(t)  = A(t) u(t) \vspace{2.5mm} \\
u(s) = u_s, 
\end{array} \right. \end{equation} 
in $X$, where $A(t) \in G(X):Y \to X$ is assumed to be an infinitesimal generator of $U(t,s)$, $-T \le t,s \le T$ is satisfied, $Y$ is a dense subspace of $X$ permitting the representation shown in Eq.~\eqref{logex}, and $u_s$ is an element of $X$.

As seen in Eq.~\eqref{replce}, under the assumption of commutation, a related Cauchy problem is obtained as
\begin{equation} \left\{  \begin{array}{ll} \label{reweq}
\partial_t v(t,s)  = (\partial_t a(t,s)) ~ v(t,s) \vspace{2.5mm} \\
v(s,s) = e^{a(s,s)} u_s,
\end{array} \right. \end{equation}
in $X$, where $\partial_t a(t,s) = \partial_t {\rm Log} (U(t,s)+\kappa I)$ is well-defined.
It is possible to solve re-written Cauchy problem, and the solution is represented by
\[  \begin{array}{ll}
v(t,s) =  e^{a(t,s)} u_s  =
 \sum_{n=0}^{\infty} \frac{ a(t,s)^n}{n!} u_s
 \end{array} \]
 for $u_s \in X$  (cf.~Eq.~\eqref{convp}).

\begin{theorem} \label{hols}
Operator $e^{a(t,s)}$ is holomorphic.
\end{theorem}

\begin{proof}
According to the boundedness of $a(t,s)$ on $X$ (Lemma~\ref{lem3}), $\partial_t^n  e^{a(t,s)}$~\cite{51taylor} is possible to be represented as
\begin{equation} \label{anreap} \begin{array}{ll}
\partial_t^n  e^{a(t,s)} = \frac{1}{2 \pi i} \int_{\Gamma} \lambda^n e^{\lambda} (\lambda-a(t,s))^{-1} ~ d \lambda, 
\end{array} \end{equation} 
for a certain $\kappa$, where $ \lambda^n e^{\lambda}$ does not hold any singularity for any finite $\lambda$.
Following the standard theory of evolution equation, 
\[  \begin{array}{ll}
 \| \partial_t^n  e^{a(t,s)}   \| \le \frac{C_{\theta,n}}{ \pi (t \sin \theta)^n}
 \end{array} \] 
is true for a certain constant $C_{\theta,n}$  ($n = 0,1,2,\cdots$), where $\theta \in (0 \pi/2)$ and $|\arg t| < \pi/2$ are satisfied (for the detail, e.g., see \cite{79tanabe}).
It follows that
\begin{equation} \label{leiq} \begin{array}{ll}
{ \displaystyle  \lim_{t \to +0} } \sup t^n \| \partial_t^n  e^{a(t,s)} \|
\le {\displaystyle \lim_{t \to +0}} \sup t^n   \frac{C_{\theta,n}}{\pi (t \sin \theta)^n}  < \infty.
\end{array} \end{equation} 
Consequently, for $|z-t|<t \sin \theta$, the power series expansion
\[ \begin{array}{ll}
 \sum_{n=0}^{\infty} \frac{(z-t)^n}{n !}   \partial_t^n e^{a(t,s)} 
 \end{array} \]
is uniformly convergent in a wider sense.
Therefore $e^{a(t,s)}$ is holomorphic.
\end{proof}

\begin{theorem} \label{reprr}
For $u_s \in X$ there exists a unique solution $u(t) \in C([-T,T];X)$ of \eqref{homoporo} with a convergent power series representation:
\begin{equation} \label{dairep} \begin{array}{ll} 
 u(t) = U(t,s) u_s =  ( e^{a(t,s)}  - \kappa  I ) u_s = \left( \sum_{n=0}^{\infty} \frac{ a(t,s)^n}{n!}  - \kappa I \right)  u_s,
\end{array} \end{equation}
where $\kappa$ is a certain complex number.
\end{theorem}

\begin{proof}
The unique existence follows from the assumption for $A(t)$.
$e^{a(t,s)}$ is holomorphic function (Theorem~\ref{hols}) with the convergent power series representation (Eq.~\eqref{convp}).
The solution of the original Cauchy problem is obtained as
 \[ \begin{array}{ll}
  u(t) =   (I+ \kappa U(s,t))^{-1}v(t,s) =  (I+ \kappa U(s,t))^{-1} \sum_{n=0}^{\infty} \frac{ a(t,s)^n}{n!} u_s
  \end{array} \]
for the initial value $u_s \in X$.
Note that $A(t)$ is not assumed to be a generator of analytic evolution family, but only a generator of invertible evolution family.
\end{proof}

For $I_{\lambda}$ denoting the resolvent operator of $A(t)$, the evolution operator defined by the Hille-Yosida approximation is written by
\[ \begin{array}{ll}
u(t) = {\displaystyle \lim_{\lambda \to 0} } \exp ( \int_s^t I_{\lambda} A(\tau) ~d \tau ) u_s,
\end{array} \] 
so that more informative representation is provided by Theorem~\ref{reprr} compared to the standard theory based on the Hille-Yosida theorem.

\subsection{Non-autonomous case}
Series representation in autonomous part leads to the enhancement of the solvability.
Let $Y$ be a dense subspace of $X$ permitting the representation shown in Eq.~\eqref{logex},  and $u_s$ is an element of $X$.
Let us consider non-autonomous Cauchy problem
\begin{equation} \label{origih} \left\{  \begin{array}{ll}
\partial_t u(t)  = A(t)  u(t) + f(t) \vspace{2.5mm} \\
u(s) = u_s
\end{array} \right. \end{equation}
in $X$, where $A(t) \in G(X):Y \to X$ is assumed to be the infinitesimal generator $U(t,s)$, $f \in L^1(-T,T;X)$ is locally H\"older continuous on $[-T,T]$
\[ \begin{array}{ll} 
\| f(t) - f(s) \| \le C_{H} |t-s|^{\gamma}
\end{array} \]
for a certain positive constant $C_{H}$, $\gamma \le 1$ and $-T \le t,s \le T$. 
The solution of non-autonomous problem does not necessarily exist in such a setting (in general, $f \in C([-T,T];X)$ is necessary).

\begin{theorem} \label{thm-inh}
Let $f \in L^1(-T,T;X)$ be locally H\"older continuous on $[-T,T]$.
For $u_s \in X$ there exists a unique solution $u(t) \in C([-T,T];X)$ for \eqref{origih} such that
\[ \begin{array}{ll} 
 u(t) = [  \sum_{n=0}^{\infty} \frac{ a(t,s)^n}{n!}  - \kappa I ]  u_s +  \int_s^t [ \sum_{n=0}^{\infty} \frac{ a(t,\tau)^n}{n!}  - \kappa I ]  f(\tau) d\tau
\end{array} \]
using a certain complex number $\kappa$.
\end{theorem}

\begin{proof}
Let us start with cases with $f \in C([-T,T];X)$.
The unique existence follows from the standard theory of evolution equation.
The representation follows from that of $U(t,s)$ and the Duhamel's principle
\begin{equation}  \begin{array}{ll}
 u(t,s)  =   U(t,s) u_s + \int_s^tU(t,\tau) f(\tau) d\tau \vspace{1.5mm}\\
 =   ( e^{a(t,s)} - \kappa I )  u_s + \int_s^t  [  e^{a(t,\tau)} - \kappa I ]  f(\tau) d\tau,
\end{array}  \end{equation}
where the convergent power series representation of $e^{a(t,s)}$ is valid (cf.~Eq.~\eqref{convp}). 

Next let us consider cases with the locally H\"older continuous $f(t)$.
According to the linearity of Eq.~\eqref{origih}, it is sufficient to consider the inhomogeneous term.
For $\epsilon$ satisfying $0 < \epsilon << T$,
\[   \begin{array}{ll}
  \int_s^{t+\epsilon} [e^{a(t,\tau)} -\kappa I] f(\tau) d\tau  
  ~\to~  
   \int_s^{t} [e^{a(t,\tau)} -\kappa I] f(\tau) d\tau
\end{array}  \]
is true by taking $\epsilon \to 0$.
On the other hand,
\begin{equation} \label{conveq}  \begin{array}{ll}
   A(t)  \int_s^{t+\epsilon} [e^{a(t,\tau)} -\kappa I] f(\tau) d\tau 
=  \int_s^{t+\epsilon} A(t) U(t,\tau) f(\tau) d\tau  \vspace{1.5mm}\\
=   \int_s^{t+\epsilon} A(t)  U(t,\tau)( f(\tau) - f(t)) d\tau  + \int_s^{t+\epsilon} A(t)  U(t,\tau) f(t) d\tau  \vspace{1.5mm}\\
=   \int_s^{t+\epsilon} A(t)  U(t,\tau) ( f(\tau) - f(t)) d\tau  - \int_s^{t+\epsilon}  \partial_{\tau} U(t,\tau) f(t) d\tau   \vspace{1.5mm}\\
=   \int_s^{t+\epsilon} A(t)  U(t,\tau) ( f(\tau) - f(t)) d\tau  -  U(t,t+\epsilon) f(t) + U(t,s) f(t)  \vspace{1.5mm}\\
=   \int_s^{t+\epsilon}  (1+\kappa U(s,t))\partial_{t} a(t,s) 
 [e^{a(t,\tau)} -\kappa I]
 ( f(\tau) - f(t)) d\tau  \vspace{1.5mm}\\
\quad -  U(t,t+\epsilon) f(t) + U(t,s)f(t), 
\end{array}  \end{equation}
where $\partial_{\tau} U(t,\tau) =  -A(\tau)  U(t,\tau)$ is utilized. 
The H\"older continuity and Eq.~\eqref{leiq} lead to the strong convergence of the right hand of Eq.~\eqref{conveq}:
\[ \begin{array}{ll} 
 A(t)  \int_s^{t+\epsilon} [e^{a(t,\tau)} - \kappa I] f(\tau) d\tau   \vspace{1.5mm}\\
  \to
 \quad  \int_s^t (1+\kappa U(s,t)) (\partial_{t} a(t,s)) [e^{a(t,\tau)} -\kappa I]  (f(\tau) - f(t)) d\tau 
  + ( U(t,s) - I) f(t)
 \end{array}  \]
(due to $\epsilon \to 0$) for $f \in L^1(0,T;X)$. 
$A(t)$ is assumed to be an infinitesimal generator, so that $A(t)$ is a closed operator from $Y$ to $X$.
It follows that
\[   \begin{array}{ll}
  \int_s^{t} [e^{a(t,\tau)} -\kappa I] f(\tau) d\tau  \in Y
\end{array}  \]
and
\[ \begin{array}{ll} 
 A(t)  \int_s^{t} [e^{a(t,\tau)} -\kappa I] f(\tau) d\tau   \vspace{1.5mm}\\
  =
  \int_s^t (1+ \kappa U(s,t)) (\partial_{t} a(t,s)) [e^{a(t,\tau)} -\kappa I]  (f(\tau) - f(t)) d\tau 
  + ( U(t,s) - I) f(t)  \in X.
 \end{array}  \]
The right hand side of this equation is strongly continuous on $[-T,T]$.
As a result,
  \[ \begin{array}{ll} 
  \partial_{t}  \int_s^{t+\epsilon} [e^{a(t,\tau)} -\kappa I] f(\tau) d\tau \vspace{1.5mm}\\
  =  [ e^{a(t,t+\epsilon)} -\kappa I ] f(t+\epsilon) +   \int_s^{t+\epsilon}  (\partial_{t} a(t,\tau)) e^{a(t,\tau)}  f(\tau) d\tau \vspace{1.5mm}\\
  \to  \quad f(t) + \int_s^{t} (1+\kappa U(\tau,t))^{-1}  A(t)  (U(t,\tau)+\kappa I)  f(\tau) d\tau \vspace{1.5mm}\\
 \qquad  =   f(t) +  \int_s^{t} A(t) U(t,\tau)   f(\tau) d\tau \vspace{1.5mm}\\
 \qquad  =   f(t) + A(t)  \int_s^{t} [e^{a(t,\tau)} -\kappa I] f(\tau) d\tau.
 \end{array}  \]
We see that $ \int_s^{t} [e^{a(t,\tau)} -\kappa I] f(\tau) d\tau$ satisfies Eq.~\eqref{origih},  and that it is sufficient to assume $f \in L^1(0,T;X)$ as H\"older continuous. 
\end{proof}

This result should be compared to the standard theory of evolution equations in which the inhomogeneous term $f$ is assumed to be continuous on $[-T,T]$.

If the inhomogeneous term $f \in L^p([0,T];X)$ is further assumed to be satisfied for $1< p < \infty$, and $Y = D(A(t)) =D(A(0))$ and $A(\cdot) \in C([0,T],{\mathcal L}(Y,X))$, Eq.~\eqref{reweq} with such an inhomogeneous term corresponds to the equation exhibiting the maximal regularity of type $L^p$~\cite{01pruess}.

\section{Concluding remark}
As for the applicability of the theory, the conditions to obtain the logarithmic representation (conditions shown in Sec.~\ref{tp-group}) are not so restrictive; indeed, they can be satisfied by $C_0$-groups generated by $t$-independent infinitesimal generators. 
The most restrictive condition to obtain Eq.~\eqref{logex} is the commutation between $A(t)$ and $U(t,s)$.
Such a commutation is trivially satisfied by $t$-independent $A(t) = A$, and also satisfied when the variable $t$ is separable  (i.e., for an integrable function $g(t)$, $A(t) = g(t) A$).
In this sense the operators specified in Theorem~\ref{thm1} correspond to a moderate generalization of $t$-independent infinitesimal generators.

\end{document}